\documentclass[11pt]{amsart}
\usepackage{txfonts}
\usepackage{amscd,amssymb,amsmath,latexsym,verbatim} 
\usepackage[all]{xy}
\usepackage{breakurl}
\usepackage{multirow}
\usepackage{enumitem}
\usepackage{bbm}
\usepackage{nicefrac}

\usepackage{import}
\usepackage{tikz}
\usepackage{caption}
\usepackage{subcaption}

\usepackage{color}
\usepackage[%backref, 
            breaklinks=true, 
            colorlinks=true,
            linkcolor=red,
            urlcolor=blue,
            citecolor=dkgreen]{hyperref}

\usepackage{fancyhdr} 
\usepackage{titlesec}
\titleformat{\section}[block]{\normalfont\large\bfseries}
{\thesection.}{0.8em}{}

\titleformat{\subsection}[runin]{\normalfont\bfseries}
{\thesubsection.}{0.5em}{}[.]

\titleformat{\subsubsection}[runin]{\normalfont\normalsize\itshape}
{\thesubsubsection.}{0.5em}{}[.]

\newcounter{first}
{\end{list}}

\definecolor{dkgreen}{rgb}{0.1,0.4,0.0}
\definecolor{dkblue}{rgb}{0,0.1,0.8}
\definecolor{dkred}{rgb}{1,0,0}

\def\define{\ensuremath{\overset{\operatorname{\scriptscriptstyle def}}=}}

{ 
 
 \begin{enumerate}}{\end{enumerate}}

\theoremstyle{plain}
\newtheorem{theorem}[subsection]{Theorem}

\theoremstyle{definition}
\newtheorem{remark}[subsection]{Remark}
\newtheorem{definition}[subsection]{Definition}

\newtheorem*{conjecture*}{Conjecture}

\newcommand{\doi}[1]
{\texttt{\href{http://dx.doi.org/#1}{\nolinkurl{doi:#1}}}}
\newcommand{\web}[1]
{\texttt{\href{#1}{\nolinkurl{#1}}}}

\title{Efficiency Axioms for simplicial complexes}
\author{Ivan Martino}
\date{\today}

\begin{document}

%from basic notions
\def\rank#1{\ensuremath{\operatorname{rk} #1}}
\def\rk{\ensuremath{\operatorname{rk}}}
\def\cork{\ensuremath{\operatorname{cork}}}

\def\face{\ensuremath{\operatorname{\mathcal F} (\Delta)}}

\def\FacetsD{\ensuremath{\operatorname{Fs} \Delta}}	
\def\Facets#1{\ensuremath{\operatorname{Fs} #1}}
\def\Facet#1#2{\ensuremath{\operatorname{Fs}_{#2} #1}}

\def\closure#1#2{\ensuremath{\operatorname{cl}_{#1} (#2)}}
\def\closureDelta#1{\ensuremath{\operatorname{cl}_{\Delta} (#1)}}

\def\Star#1#2{\ensuremath{\operatorname{Star}_{#2} #1 }}
\def\StarNoD#1{\ensuremath{\operatorname{Star} #1 }}
\def\StarD#1{\ensuremath{\operatorname{Star}_{\Delta} #1 }}

\def\Link#1#2{\ensuremath{\operatorname{Link}_{#2} #1 }}
\def\LinkNoD#1{\ensuremath{\operatorname{Link} #1 }}
\def\LinkD#1{\ensuremath{\operatorname{Link}_{\Delta} #1 }}

\def\ffD{\ensuremath{\operatorname{\textbf{f}}(\Delta)}}
\def\ff#1{\ensuremath{\operatorname{\textbf{f}}(#1)}}
\def\ffi#1#2{\ensuremath{\operatorname{f}_{#1}(#2)}}

\def\charFun{\ensuremath{\mathbb{R}^{\Delta} }}

%from efficienccy
\def\vtot{\ensuremath{\operatorname{v}_{\Delta}}}
\def\vtotof#1{\ensuremath{\operatorname{v}_{#1}}}

%from shapley
\def\Coalition#1#2{\ensuremath{\operatorname{Coalition}_#2 #1 }}

\begin{abstract}
We study the notion of efficiency for cooperative games on simplicial complexes. 
In such games, the grand coalition $[n]$ may be forbidden, and, thus, it is a non-trivial problem to study the total number of payoff $\vtot$ of a cooperative game $(\Delta, v)$.

We address this question in the more general setting, by characterizing the individual values that satisfy the general efficient requirement, that is $\vtot^{gen}=\sum_{T\in \Delta} a_T v(T)$ for a generic assignment of real coefficients $a_T$.
The traditional and the probabilistic efficiency are treated as a special case of this general efficiency.

Finally, we introduce a new notion of efficiency arising from the combinatorial and topological property of the simplicial complex $\Delta$. The efficiency in this scenario is called \emph{simplicial} and we characterize the individual values fulfilling this constraint.
\end{abstract}

\maketitle

In the traditional $n$-person game the characteristic function  $v:2^{[n]}\rightarrow \mathbb{R}$ determines the \emph{worth} of each coalition, where $[n]\define \{1,\dots, n\}$.
The individual value $\phi_i$ associated to such cooperative game $(n, v)$ measures the contribution of the player $i$ in the game. We collect such values all together in the group value $\phi=(\phi_1, \phi_2, \dots, \phi_n)$. The assessment is optimistic (w. r. to $v$) if the sum of the payoff vector $\sum_i \phi_i(v)$ is greater than the $v([n])$, the worth of the grand coalition. If the contrary happens, then $\phi$ is pessimistic (w. r. to $v$).

Consider the vector space of cooperative games $\mathbb{R}^{2^n-1}$, that is the set of all characteristic function under the constrain $v(\emptyset)=0$. Provided a game $(n, v)$ in $\mathbb{R}^{2^n-1}$, we consider the scaled game $(n, c\cdot v)$ given by the characteristic function $(c\cdot v)(T)=c v(T)$ where $c$ is a real constant. 
It seems natural to assume that in this case the individual value is also scaled by $c$, $\phi_i(c\cdot v)=c\phi_i(v)$. 
Therefore consider a cone of games $\mathfrak{I}$ in $\mathbb{R}^{2^n-1}$.

We are interesting in group values that are nor optimistic or pessimistic, despite this might be an artificial condition.
This consideration leads to the constraint known as called \emph{Efficiency Axiom}:

\vspace{0.2cm}
\hspace{0.2cm} \textbf{Efficiency Axiom.} 

\hspace{0.2cm} For every cooperative game $(2^{[n]}, v)$ in $\mathfrak{I}$, one has $\sum_{i=1}^n \phi_i(v)=v([n])$.
\vspace{0.2cm}

If certain coalitions are forbidden (take for instance the grand coalition), it is necessary to study what could take the place of the the total number of payoff $\vtot$, that in the traditional case reduces to $\vtot=v([n])$.
In this work we focus on the specific instance of this problem for cooperative game on a simplicial complex \cite{Martino-cooperative, Martino-Probabilistic-value}.
We are going to shortly present our new results after introducing this new generalization for cooperative games.
%,Martino-Shapley-values}.\cite{Martino-Efficiency,Martino-Probabilistic-value}.%,Martino-cooperative,Martino-Shapley-values}.

\subsection*{Cooperative game on simplicial complex}
Inspired by several articles  \cite{Shapley-matroids-static, Shapley-matroids-dynamic, MR3886659, MR2847360, MR2825616, MR1436577, MR1707975}, the author has defined cooperative games on simplicial complexes \cite{Martino-cooperative}.
In fact, a simplicial complex is a family $\Delta$ of subsets of $[n]$ under the constrain that every subset of $X\in \Delta$, also belongs to the family $\Delta$. 
A cooperative game on $\Delta$ is defined by a characteristic function $v$:
\[
	v: \Delta \rightarrow \mathbb{R}
\]
with $v(\emptyset)=0$.
In such game, a player $i$ in $[n]$ may join a coalition $T$ only if $T\cup i\in \Delta$. In such case, the coalition is \emph{feasible}.
If every subset of $[n]$ is feasible, then $\Delta=2^{[n]}$ and we recover the literature case.

\noindent
As in the classical case, the individual value function $\phi_i(v)$ for the player $i$ determines the worth of the participation of $i$ in a \emph{feasible} coalition during the cooperative game $(\Delta, v)$.
As before, we may consider a cone $\mathfrak{I}$ of cooperative games defined in \charFun, the vector space of all cooperative games on $\Delta$.

\subsection*{Quasi-probabilistic values}

The issue of studying $\vtot$ already arises in the work of Bilbao, Driessen, Jim\'{e}nez Losada and Lebr\'{o}n \cite{Shapley-matroids-static}, where they introduce the notion of \emph{probabilistic} efficiency for games over a matroid. One of the perks of matroids is that they are pure simplicial complex; in other words the maximal facets of a matroid have the same cardinality, see for instance \cite{Stanley2012b, Stanley1996a, Stanley1991a, MR782306, Oxley, Martino2018, Borzi-Martino-D-matroids}.
Here we present the natural generalization to simplicial complexes that already appears in Section 6 of \cite{Martino-cooperative}:

\vspace{0.2cm}
\hspace{0.2cm} \textbf{Probabilistic Efficiency Axiom.} 

\hspace{0.2cm} For every cooperative game $(\Delta, v)$ in $\mathfrak{I}$, $\sum_{i=1}^n \phi_i(v)=\sum_{F\in \FacetsD}c_F v(F)$,

\noindent
\hspace{0.6cm} 
 with $\sum_{F\in \FacetsD}c_F =1$ and $c_F\geq 0$ for every facet $F$ of $\Delta$.
\vspace{0.2cm}

\noindent
In the above equation, $\FacetsD$ is the set of facets of $\Delta$, that are maximal elements by inclusion.

It is worth to recall that the Efficient axiom, and respectively the Probabilistic efficient are crucial to characterizes the Shapley values \cite{Shapley-a-value, Shapley-core-convex, Weber-robabilistic-values-for-games}, and respectively the quasi-probabilistic values that can be written as sum of Shapley values \cite{Shapley-matroids-static, Martino-cooperative}.

\subsection*{Our Approach}

We adopt a differ point of view originated by the next consideration. 
The first axiom in the theory of probabilistic values is the \emph{Linearity Axiom} for individual values,
see Section 3 of \cite{Weber-robabilistic-values-for-games}.
We are also going to assume that the total number of payoff $\vtot$ is a linear function:
\[
	\vtot: \mathbb{R}^\Delta \rightarrow \mathbb{R}.
\]
Therefore, the total number of payoff can be written as 
\begin{equation*}
	\vtot=\sum_{T\in \Delta} a_T v(T).
\end{equation*}
(Note that we allow $T=\emptyset$, because $v(\emptyset)=0$.)
The choice of the coefficients $a_T$ describes a diverse efficiency scenario: for instance, the \emph{efficiency axiom} for $(2^{[n]}, v)$ is given by setting $a_T=0$ for every $T\neq [n]$ and $a_{[n]}=1$.
Similarly, the \emph{probabilistic efficiency} of Bilbao, Driessen, Jim\'{e}nez Losada and Lebr\'{o}n \cite{Shapley-matroids-static} is obtained from the choice of $a_F=c_F$ for every facet $F$ of $\Delta$ and $a_T=0$, otherwise. 
Furthermore, a more flexible efficiency condition is needed in Theorem E of \cite{Martino-Probabilistic-value} to properly encode Shapley values on simplicial complexes.

Our first result characterize individual values that satisfies the \emph{general} efficient condition.

\setcounter{section}{2}
\begin{theorem}
Let $\Delta$ be a simplicial complex and let $\mathfrak{I}$ be a cone of cooperative games $v$ defined on $\Delta$ containing the carrier games $\mathcal{C}$ and $\hat{\mathcal{C}}$.

Let $\phi$ be a group value on $\mathcal{I}$ such that for each $i\in [n]$ and assume that for each $v\in\mathfrak{I}$, we can write:
\[
		\phi_i(v)=\sum_{T\in \Link{i}{\Delta}}p_T^i \left(v(T\cup i) - v(T)\right).
\]

\noindent
The group value $\phi$ satisfies the simplicial efficiency axiom if and only if for all non-facet $T$ in $\Delta$
\begin{equation*}%\label{eq:first-condition-efficienc-generic}	
\sum_{i\in T}p^i_{T\setminus i}-\sum_{j, T\in \Link{j}{\Delta}}p_T^j=a_T,
\end{equation*}
and for every facet $F$ of $\Delta$
\begin{equation*}%\label{eq:second-condition-efficiency-generic}
	\sum_{i\in F}p^i_{F\setminus i}=a_F.
\end{equation*}
\end{theorem}

In Theorem \ref{thm:t.e.a.}, we specialize the previous result in the traditional setting and we reproduce 
Theorem 11 of \cite{Weber-robabilistic-values-for-games} 
%Theorem \ref{thm:probabilistic-characterization}

\noindent
Similarly, applying Theorem \ref{thm:g.e.a.}, we characterize the individual values that satisfy the probabilistic efficiency, see Theorem \ref{thm:probabilistic-characterization}.

Moreover in Section \ref{sec:simplicial-efficiency}, we introduce and study a new efficiency conditions: the \emph{simplicial} efficiency.

\subsection*{The simplicial efficiency}
In this scenario is built on the following point of view.
If the grand coalition $[n]$ is forbidden, the largest possible coalitions are precisely the elements of \FacetsD.
However, the facets may intersect.  
Thus, in Section \ref{sec:simplicial-efficiency}, with an inclusion-exclusion argument we study the following number of total payoff:
	\[
		\vtot^{simpl}=\sum_{l=1}^{k}\sum_{\{F_{i_j}\}\in{\FacetsD \choose l}} (-1)^{l+1}v(F_{i_1}\cap \dots \cap F_{i_l}).
	\]
We classify the values fulfilling this requirement in Theorem \ref{thm:s.e.a.}. In the specific case of matroids, we show that $\vtot^{simpl}$ only depends by the dimension-zero and codimension-one skeleton of $\Delta$, see Theorem \ref{thm-reduction-to-matroids}.
Moreover in Remarks \ref{rmk:connection-2-matroids-probabilistic} and \ref{rmk:connection-2-matroids-approximation} we argue that the probabilistic efficiency is a first-hand approximation of the simplicial efficiency.

%\subsection*{Shapley condition} In \cite{Martino-Probabilistic-value} we show that one can define a \emph{Shapley-like} values for simplicial complexes satisfying certain topological condition. In this case, one could characterize the values through a specific efficiency request. The payoff function is in this case
%	\[
%		\vtot^{Sh}=....
%		%https://core.ac.uk/download/pdf/7151113.pdf
%	\]
%	In Section \textbf{REF} we develop this possibilities.
%
%

%\begin{theorem}\label{thm:vtot-grothendieck}
%	Let $k$ be the number of facets of a simplicial complex $\Delta$, $k\define \#\FacetsD$.
%	%
%	The total number of payoff $\vtot$ for the cooperative game $(\Delta, v)$ is precisely
%	\begin{equation}\label{eq:vtot-grothendieck}
%		\vtot=\sum_{l=1}^{k}\sum_{\{F_{i_j}\}\in{\FacetsD \choose l}} (-1)^{l+1}v(F_{i_1}\cap \dots \cap F_{i_l})
%	\end{equation}
%%	\[
%%		\vtot=\sum_{k=1}^n\sum_{F_{i_1}, \dots, F_{i_k}\in{\FacetsD \choose k}} (-1)^{k+1}v(F_{i_1}\cap \dots, \cap F_{i_k})
%%	\]
%	where if $A$ is a finite set, ${A \choose l}$ is the set of subsets of $A$ of cardinality $l$ and $F_{i_j}$ is the $i_j$-th elements of the set $\{F_{i_1}, \dots, F_{i_l}\}$.
%	
%	
%	Moreover, if $\Delta$ is the full simplex $2^{[n]}$, then $\vtot=v([n])$.
%\end{theorem}
%

\subsection*{Acknowledgments} 
	The author is currently supported by the Knut and Alice Wallenberg Foundation and by the Royal Swedish Academy of Science.

\setcounter{section}{0}

\section{Preliminaries}\label{sec:preliminaries}

In this manuscript $n$ is a positive integer and we denote by $[n]\define\{1,\dots, n\}$. 
A (finite) simplicial complex $\Delta$ over $n$ verticies is a family of subsets of $[n]$ with the simplicial condition:
\[
	T\in \Delta, S\subseteq T\Rightarrow S\in \Delta.
\]
The set of elements of the simplicial complex is denoted by $\operatorname{Set}(\Delta)$.
A facet of $\Delta$ is a set $F$ in $\Delta$ that is maximal by inclusion.
The set of facets is $\FacetsD$. If every facets have the same cardinality then the simplicial complex is said to be \emph{pure}.

In this paper, we refer to results of \cite{Martino-cooperative} and \cite{Martino-Probabilistic-value} that use the notions of star and link of a vertex $i$ in $\Delta$. We recall those for completeness, even if they will not be central in our study.

If $S$ is an element in $\Delta$, then $\bar{S}\define 2^{S}$ is the $(|S|-1)$-dimensional simplex defined on the verticies of $S$.

\begin{definition}\label{def:star}
The star of an element $S$ in $\Delta$ is the simplicial complex defined to be the collection of all subset in $\bar{T}$ with $T$ being in $\Delta$ ans containing $S$,
\begin{equation*}%\label{eq:star-def}
	\Star{S}{\Delta}=\{A:\, A\in \bar{T},\, T\in \Delta,\, S\subseteq T\}.
\end{equation*}
\end{definition}

\noindent
We highlight when $S=\{i\}$ is a vertex, then $\Star{i}{\Delta}$ is the set of simplex $\bar{T}$ containing $i$, that is
\[
	\Star{i}{\Delta}=\{A:\, A \subseteq T,\, i\in T\in \Delta\}.
\]

\begin{definition}\label{def:link}
The link of of an element $S$ in a simplicial complex $\Delta$ is made by the subsets $A$ of $T\in \Delta$, such that $T$ is disjoint by $S$ and can be completed by $S$, $S\cup T$, to an element in $\Delta$:

\begin{equation*}%\label{eq:link-def}
	\Link{S}{\Delta}=\{A:\, A\in \bar{T} \mbox{ with } T\in \Delta \mbox{ such that } S\cap T = \emptyset, S\cup T \in \Delta\}.
\end{equation*}
\end{definition}

\noindent
The case when $S$ is the singleton $\{i\}$ will be extremely relevant in our work: $\Link{i}{\Delta}$ is the set of simplex $T$ in $\Delta$ with $i\notin T$ such that $T\cup i\in \Delta$:
\[
	\Link{i}{\Delta}=\{T\in \Delta: i\notin T \mbox{ and } T\cup i \in \Delta\}.
\]
%we often write $\LinkNoD{i}$, instead of $\Link{i}{\Delta}$.

\subsection{Matroids}
Since they were introduced by Whitney \cite{Whi35} in 1935, matroids are at the crossroads of Algebra, Combinatorics, Geometry, and Topology.  
New variations have appeared in literature encoding different type of independence \cite{Moci2012, Martino2018,  Borzi-Martino-D-matroids}, but here we recall the traditional one and we refer for a detailed description to \cite{Stanley2012b, Stanley1996a, Stanley1991a, MR782306, Oxley}.
	
	A \emph{matroid} on the ground set $[n]$ is a collection $\mathcal{I}$ of subsets of $[n]$ (called \emph{independent sets}), such that
\textbf{(I1)} $\emptyset \in \mathcal{I}$, \textbf{(I2)} $A \subseteq B \in \mathcal{I} \Rightarrow A \in \mathcal{I}$, and \textbf{(I3)} $A,B \in \mathcal{I}$, $|A| < |B|$ $\Rightarrow$ $\exists b \in B \setminus A: A \cup \{b\} \in \mathcal{I}$.

\noindent
The first two axioms make $\mathcal{I}$ into a (non-empty) simplicial complex. Axiom (I3) is sometimes referred as \emph{independent set exchange property} (or \emph{independence augmentation axiom.}). Let $\mathcal{I}$ be a matroid on the ground set $[n]$, and let $A \subseteq [n]$. All maximal independent subsets of $A$ have the same cardinality, called the \emph{rank} $\rk(A)$ of $A$, whereas the \emph{corank} of $A$ is $\cork(A) = \rk([n]) - \rk(A)$. Hence a matroid is a pure simplicial complex, see for instance Figure \ref{fig:figure-matroid-yes}.
Not every pure simplicial complex is a matroid; in fact Figure \ref{fig:figure-matroid-no} shows a rank three simplicial complex that is not a matroid. Indeed, the independent set $\{5\}$, cannot be extended to a base by any element in the independent set $\{2,3\}$.
Simplicial complexes that are not matroids are extremely important in Mathematics; few example can be found in \cite{Martino2015e, MR3503390, Martino-Greco-pinched}.

%
%A matroid can be equivalently defined by assigning the (co)rank of the subsets of $[n]$ (see \cite[Corollary 1.3.4]{Oxley}).

\begin{figure}
\centering
\begin{subfigure}[b]{5cm}
\begin{tikzpicture}[scale=0.70]
   % NODES
   	\coordinate (F) at ( -3.0,  1.0);
  	\coordinate (S) at (-2.5, -2.0); 
	\coordinate (FB) at ( 0.0, -0.0); 
	\coordinate (TV) at (3.0, 1.0); 
	\coordinate (E) at (2.5, -2.0);
	 
   \node [left] at ( -3.0,  1.0) {1}; 
   \node [left] at (-2.5, -2.0) {2}; 
	\node [above] at ( 0.0, -0.0) {3}; 
   \node [right] at (3.0, 1.0) {4}; 
   \node [right] at (2.5, -2.0) {5};

   % DRAW TREE
   \filldraw [draw=black, fill=red!20, line width=1.5pt] (FB)--(TV)--(E) -- (FB);
   
   \filldraw [draw=black, fill=green!20, line width=1.5pt] (FB)--(S)--(E) -- (FB);

\filldraw [draw=black, fill=blue!20, line width=1.5pt] (FB)--(F)--(S) -- (FB); 
   
 %  \draw [draw=black, line width=1.5pt] (S)--(F);
   
  % \draw [draw=black, line width=1.5pt] (FB)--(F);
\end{tikzpicture}
        \subcaption{This simplicial complex is a matroid.}
        \label{fig:figure-matroid-yes}
    \end{subfigure}
~\hspace{0.5cm}
	\begin{subfigure}[b]{5cm}
\begin{tikzpicture}[scale=0.75]
   % NODES
   	\coordinate (F) at ( -3.0,  1.0);
  	\coordinate (S) at (-2.5, -2.0); 
	\coordinate (FB) at ( 0.0, -0.0); 
	\coordinate (TV) at (3.0, 1.0); 
	\coordinate (E) at (2.5, -2.0);
	 
   \node [left] at ( -3.0,  1.0) {1}; 
   \node [left] at (-2.5, -2.0) {2}; 
	\node [above] at ( 0.0, -0.0) {3}; 
   \node [right] at (3.0, 1.0) {4}; 
   \node [right] at (2.5, -2.0) {5};

   % DRAW TREE
   \filldraw [draw=black, fill=red!20, line width=1.5pt] (FB)--(TV)--(E) -- (FB);

   \filldraw [draw=black, fill=blue!20, line width=1.5pt] (FB)--(S)--(F) -- (FB);
   
%   \draw [draw=black, line width=1.5pt] (S)--(F);
\end{tikzpicture}
        \subcaption{This simplicial complex is not a matroid.}
        \label{fig:figure-matroid-no}
	\end{subfigure}

\caption{Two examples of pure simplicial complexes.}
\end{figure}
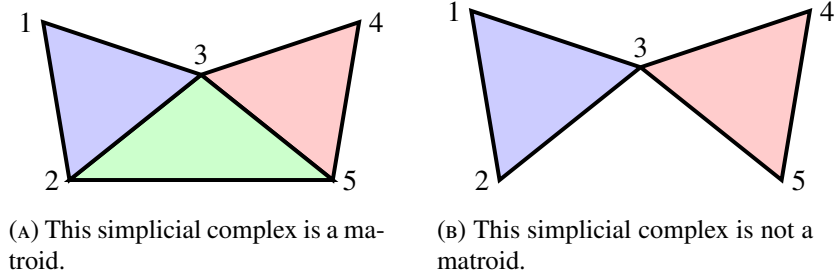

%%%%%%%%%%%%%%%%%%%%
%%%%%%%%%%%%%%%%%%%%
%%%%%%%%%%%%%%%%%%%%
%%%%%%%%%%%%%%%%%%%%
%%%%%%%%%%%%%%%%%%%%
%%%%%%%%%%%%%%%%%%%%
%%%%%%%%%%%%%%%%%%%%
%%%%%%%%%%%%%%%%%%%%
%%%%%%%%%%%%%%%%%%%%
\subsection{Cooperative games on simplicial complexes}
In \cite{Martino-cooperative}, the author introduces the notion of cooperative game on the simplicial complex $\Delta$, inspired by the work of Bilbao, Driessen, Jim\'{e}nez Losada and Lebr\'{o}n \cite{Shapley-matroids-static}.

Here we shortly recall that a cooperative game on a simplicial complex $\Delta$ is the pair $(\Delta, v)$ where $v$ is a characteristic function $v:\operatorname{Set}(\Delta)\rightarrow \mathbb{R}$ under the constrain $v(\emptyset)=0$. 
The verticies of $\Delta$ are the players of the cooperative game and a coalition $T$ is \emph{feasible} if $T\in \Delta$.
The set $\charFun$
%$R(\Delta)$ 
of characteristic functions on $\Delta$ is naturally a real vector space.

Given a cooperative game $(\Delta, v)$, one can re-scale the characteristic function with a scalar $c$ and obtain a new cooperative game $(\Delta, cv)$, where $(cv)(T)=c(v(T))$ for every subset $T\in \Delta$.

\begin{definition}
An \emph{individual value} for a player $i$ in $[n]$ is a function $\phi_i:\charFun\rightarrow \mathbb{R}$.
\end{definition}

The goal of each individual value $\phi_i(v)$ is assessing the worth of the participation of the player $i$ in to the game. 
Naturally, we are looking for values such that $\phi_i(cv)=c\phi_i(v)$, because the worth of each player is just re-scaled.
For this reason we often consider a cone $\mathfrak{I}$ of cooperative game in $\charFun$.

\subsection{Efficiency axioms previously introduced in literature}\label{subsec:intro-efficiency}

In the traditional case the sum of such values is often compare with the total number of payoff $v([n])$ of the grand coalition $[n]$. 
If the sum $\sum_i \phi_i(v)$ is greater than $v([n])$, then the assessment is going to distribute to all the player a larger amount than the one eventually obtained. This is the optimistic (w. r. to $v$) setting. Vice versa, the group value $\phi(v)=\{\phi_1(v), \dots, \phi_n(v)\}$ is pessimistic.

It may be artificial, but it is surely interesting to consider group values that are not optimistic or pessimistic.
This leads to the so called \emph{Efficiency Axiom}:

\vspace{0.2cm}
\hspace{0.2cm} \textbf{Efficiency Axiom.} 

\hspace{0.2cm} For every cooperative game $(2^{[n]}, v)$ in $\mathfrak{I}$, $\sum_{i=1}^n \phi_i(v)=v([n])$.
\vspace{0.2cm}

This axiom is part of the requirement in the characterization of the Shapley values \cite{Shapley-a-value, Shapley-core-convex, Weber-robabilistic-values-for-games}.

The effort of generalizing such requirement can be found already in the work \cite{Shapley-matroids-static}, where they deal with the notion of \emph{probabilistic} efficiency for a cooperative game over a matroid. Here we present the natural generalization to simplicial complexes: 

\vspace{0.2cm}

\hspace{0.2cm} \textbf{Probabilistic Efficiency Axiom.} 

\hspace{0.2cm} For every cooperative game $(\Delta, v)$ in $\mathfrak{I}$, $\sum_{i=1}^n \phi_i(v)=\sum_{F\in \FacetsD}c_F v(F)$, 

\noindent
\hspace{0.6cm} 
 with $\sum_{F\in \FacetsD}c_F =1$ and $c_F\geq 0$ for every facet $F$.
\vspace{0.2cm}

As the the previous case, also this axiom is used as a necessary and sufficient condition for writing
the quasi-probabilistic values as sum of Shapley values \cite{Shapley-matroids-static, Martino-cooperative}.

\subsection{Carrier games}
There are two set of games that have a crucial role in the theory of probabilistic values \cite{Weber-robabilistic-values-for-games}. 
We are going to called both set \emph{carrier games}, even if in literature this terminology often refer to the first one:
\[
	\mathcal{C} = \{v_T: \emptyset \neq T \subset [n] \}, \,\,\,\hat{\mathcal{C}} = \{\hat{v}_T: \emptyset \neq T \subset [n] \},
\]
where $v_T$ and $\hat{v}_T$ are so defined:
\[
	v_T (S)=\begin{cases}
				1 &  T\subseteq S\\
				0 &\text{otherwise.}
			\end{cases},\,\,\, \hat{v}_T (S)=\begin{cases}
				1 &  T\subsetneq S\\
				0 &\text{otherwise.}
			\end{cases}
\]
We generalize these notation for any element $T$ of a simplicial complex. Indeed for every partially order set $(P, \leq_P)$ and every element $q$ in $P$ we consider the following function:
\[
	u_q^P(s)\define \begin{cases}
				1 &  q\leq_P s \\
				0 &\text{otherwise.}
			\end{cases},\,\,\, 				\hat{u}_q^P(s)\define \begin{cases}
				1 &  q<_P s \\
				0 &\text{otherwise.}
			\end{cases}
\]
Thus, we define
\[
	v_T (S)\define u_T^\Delta(S),\,\,\, \hat{v}_T (S)\define \hat{u}_T^\Delta (S).
\]
It is easy to see that in the classical case (when $\Delta$ is a full simplex on $n$ verticies) these functions reproduce the carrier games.

\begin{definition}
	Let $\Delta$ be a simplicial complex. The sets of carrier games are so defined:
	\[
		\mathcal{C} = \{v_T: \emptyset \neq T \in \Delta \}, \,\,\,\hat{\mathcal{C}} = \{\hat{v}_T: \emptyset \neq T\in \Delta \}, 
	\]
	where $v_T (S)\define u_T^\Delta(S)$ and $\hat{v}_T (S)\define \hat{u}_T^\Delta (S)$; moreover, $\hat{v}_{\emptyset}\define \hat{u}_{\emptyset}^\Delta$.
\end{definition}

%%%%%%%%%%%%%%%%%%%%
%%%%%%%%%%%%%%%%%%%%
%%%%%%%%%%%%%%%%%%%%
%%%%%%%%%%%%%%%%%%%%
%%%%%%%%%%%%%%%%%%%%
%%%%%%%%%%%%%%%%%%%%
%%%%%%%%%%%%%%%%%%%%
%%%%%%%%%%%%%%%%%%%%
%%%%%%%%%%%%%%%%%%%%

\section{A unique approach to efficiency}
Let us assume that the total number of payoff is prescribed as 
\begin{equation}\label{eq:total-number-payoff-generic}
	\vtot^{gen}=\sum_{T\in \Delta} a_T v(T).
\end{equation}
%\[
%
%\]
and we require the following efficiency constrain:

\vspace{0.2cm}
\hspace{0.2cm} \textbf{Generic Efficiency Axiom.} 

\hspace{0.2cm} For every cooperative game $(2^{[n]}, v)$ in $\mathfrak{I}$, $\sum_{i=1}^n \phi_i(v)=\vtot^{gen}$.
\vspace{0.2cm}

Then, we characterize the individual values that can be written in the classical sum of marginal contributions and that satisfy the \emph{Generic Efficiency Axiom}.

\begin{theorem}\label{thm:g.e.a.}
Let $\Delta$ be a simplicial complex and let $\mathfrak{I}$ be a cone of cooperative games $v$ defined on $\Delta$ containing the carrier games $\mathcal{C}$ and $\hat{\mathcal{C}}$.

\noindent
Let $\phi$ be a group value on $\mathcal{I}$ such that for each $i\in [n]$ and assume that for each $v\in\mathfrak{I}$, we can write:
\[
		\phi_i(v)=\sum_{T\in \Link{i}{\Delta}}p_T^i \left(v(T\cup i) - v(T)\right).
\]

The group value $\phi$ satisfies the simplicial efficiency axiom if and only if for all non-facet $T$ in $\Delta$
\begin{equation}\label{eq:first-condition-efficienc-generic}	
\sum_{i\in T}p^i_{T\setminus i}-\sum_{j, T\in \Link{j}{\Delta}}p_T^j=a_T,
\end{equation}
and for all facet $F$ of $\Delta$
\begin{equation}\label{eq:second-condition-efficiency-generic}
	\sum_{i\in F}p^i_{F\setminus i}=a_F.
\end{equation}
\end{theorem}
\begin{proof}
	Let us show that the equations \eqref{eq:first-condition-efficienc-generic} and \eqref{eq:second-condition-efficiency-generic} are necessary. 	%Let $\sum\phi$ be the sum of the individual value $\phi_i(v)$ for all players. 
Using our assumption, this is:
	\[
		\sum_{i=1}^n \phi_i(v)=\sum_{i\in [n]} \sum_{T\in \Link{i}{\Delta}}p_T^i \left(v(T\cup i) - v(T)\right).
	\]
	We reorder the terms in the sum as
	\[
		\sum_{i=1}^n \phi_i(v)=\sum_{T\in \Delta}v(T)\left(\sum_{i\in T}p^i_{T\setminus i}-\sum_{j, T\in \Link{j}{\Delta}}p_T^j\right).
	\]
	Since $\vtot^{gen}=\sum_{T\in \Delta}a_T v(T)$, when $T=F$ is a facet, we get \eqref{eq:second-condition-efficiency-generic}.
	If $T$ is not a facet, then it is clear that $\phi$ satisfies \eqref{eq:first-condition-efficienc-generic}.
	
%	Staring at the same equations, but on the opposite direction shows also the only if part.
	
	To prove the opposite direction one needs to note that if $T$ is not a facet, then
	\[
		\sum_{i=1}^n\phi(v_T)-\sum_{i=1}^n\phi(\hat{v}_T)=	\sum_{i\in T}p^i_{T\setminus i}-\sum_{j, T\in \Link{j}{\Delta}}p_T^j.
	\]
	By hypothesis the latter equals $a_T$.
	Similarly if $F$ is a facet then
	\[
		\sum_{i=1}^n\phi(v_F) = \sum_{i\in T}p^i_{T\setminus i}.
	\]
	and by hypothesis the latter is equal $a_F$.
\end{proof}

\subsection{The Traditional Efficiency}
As a corollary of the previous theorem we can easily obtain Theorem 11 of \cite{Weber-robabilistic-values-for-games}.
In fact, in the traditional cooperative game on a full simplicial complex with $n$ verticies, the efficiency axiom is the following:

\vspace{0.2cm}
\hspace{0.2cm} \textbf{Efficiency Axiom.} 

\hspace{0.2cm} For every cooperative game $(2^{[n]}, v)$ in $\mathfrak{I}$, one has $\sum_{i=1}^n \phi_i(v)=v([n])$.
\vspace{0.2cm}

\noindent
Thus, in equation \eqref{eq:total-number-payoff-generic}, $a_{T}=0$ for every subset of $[n]$, but $a_{[n]}=1$.
Hence, Theorem 11 of \cite{Weber-robabilistic-values-for-games} follows as a corollary of our result:

\begin{theorem}\label{thm:t.e.a.}
Let $\mathfrak{I}\subset \mathbb{R}^{2^n-1}$ be the cone of cooperative games containing the (classical) carrier games $\mathcal{C}$ and $\hat{\mathcal{C}}$.

Let $\phi$ be a group value on $\mathcal{I}$ such that for each $i\in [n]$ and assume that for each $v\in\mathfrak{I}$, we can write:
\[
		\phi_i(v)=\sum_{T\in [n]\setminus i}p_T^i \left(v(T\cup i) - v(T)\right).
\]

\noindent
The group value $\phi$ satisfies the simplicial efficiency axiom if and only if for all non-facet $T$ in $\Delta$
\begin{equation*}%\label{eq:first-condition-efficienc-generic}	
\sum_{i\in T}p^i_{T\setminus i}-\sum_{j\notin T}p_T^j=0,
\end{equation*}
and for all facet $F$ of $\Delta$
\begin{equation*}%\label{eq:second-condition-efficiency-generic}
	\sum_{i\in [n]}p^i_{[n]\setminus i}=1.
\end{equation*}
\end{theorem}
\begin{proof}
	Observe that $\Link{i}{2^{[n]}}=[n]\setminus i$ and there is only one facet, $F=[n]$. Then, apply Theorem \ref{thm:g.e.a.} with $\vtot=v([n])$.
\end{proof}

\subsection{The Probabilistic Efficiency}

We can also characterize the individual values that satisfy the probabilistic efficiency, introduced in \cite{Shapley-matroids-static} for cooperative games on matroids and generalized for games on every simplicial complex in in Section 6 of \cite{Martino-cooperative}:

\vspace{0.2cm}
\hspace{0.2cm} \textbf{Probabilistic Efficiency Axiom.} 

\hspace{0.2cm} For every cooperative game $(\Delta, v)$ in $\mathfrak{I}$, $\sum_{i=1}^n \phi_i(v)=\sum_{F\in \FacetsD}c_F v(F)$, 

\noindent
\hspace{0.6cm} 
 with $\sum_{F\in \FacetsD}c_F =1$ and $c_F\geq 0$ for every facet $F$.
\vspace{0.2cm}

\noindent
Thus, the total number of payoff can be written as in equation \eqref{eq:total-number-payoff-generic} by setting $a_{T}=0$ for every subset of $[n]$, but the facets where $a_{F}=c_F$.
In next results we characterize the individual values that fulfill this efficiency request.

\begin{theorem}\label{thm:probabilistic-characterization}
Let $\Delta$ be a simplicial complex and let $\mathfrak{I}$ be a cone of cooperative game defined on $\Delta$ containing the carrier games $\mathcal{C}$ and $\hat{\mathcal{C}}$.

\noindent
Let $\phi$ be a group value on $\mathfrak{I}$ such that for each $i\in [n]$ and each $v\in\mathfrak{I}$, we can write:
\[
		\phi_i(v)=\sum_{T\in \Link{i}{\Delta}}p_T^i \left(v(T\cup i) - v(T)\right).
\]

The group value $\phi$ satisfies the probabilistic efficiency axiom if and only if for all non-facet $T$ in $\Delta$
\begin{equation}\label{eq:first-condition-efficiency-probabilistic}
	\sum_{i\in T}p^i_{T\setminus i}-\sum_{j, T\in \Link{j}{\Delta}}p_T^j=0,
\end{equation}
and for all facet $F$ of $\Delta$
\begin{equation}\label{eq:second-condition-efficiency-probabilistic}
	\sum_{i\in F}p^i_{F\setminus i}=c_F
\end{equation}
%\begin{itemize}
%	\item[i)]
%	\item[ii)]
%\end{itemize}
\end{theorem}
\begin{proof}
	Simply apply Theorem \ref{thm:g.e.a.} with $\vtot=\sum_{F\in \FacetsD}c_F v(F)$.
\end{proof}

\begin{remark}
	We note that the conditions $\sum_{F\in \FacetsD}c_F =1$ and $c_F\geq 0$ are irrelevant in the Theorem \ref{thm:probabilistic-characterization}. Thus the statement holds also for the individual values that satisfy the condition $\sum_{i=1}^n \phi_i(v)=\sum_{F\in \FacetsD}c_F v(F)$ without any requirement on $c_F$. For instance, $c_F$ may be a negative number.
\end{remark}

\section{The Simplicial Efficiency}\label{sec:simplicial-efficiency}

In this section we want to introduce a new concept of efficiency that differs from the previous approaches.
Since the grand coalition $[n]$ is forbidden, the largest possible coalitions are precisely the elements of \FacetsD, but the facets may intersect. If we work under the constrain that $\Delta$ is a connected simplicial complex, if $\Delta\neq 2^{[n]}$, then every facet intersect at least another facet.
This reasoning leads to the inclusion exclusion problem for computing the total number of payoff $\vtot$ for the cooperative game $(\Delta, v)$.
Let us set up a proper arithmetic for doing this.

\vspace{0.2cm}

Let  $F_1, F_2, ..., F_k$ be a random order of the facets. 
We begin by considering $v(F_1)+v(F_2)$. If they intersect, then the worth of intersection is double counted and we subtract this: $v(F_1)+v(F_2)-v(F_1\cap F_2)$. 
Next, we add $v(F_3)$ and we might need to correct again our computation subtracting $v(F_3\cap (F_1\cup F_2))$. 
Now, $v(F_3\cap (F_1\cup F_2))$ may not be a simplex, but it is  the union of simplicies. Therefore, we use the following trick: let $K$ be a sub-complex of $\Delta$, then 
\begin{equation}\label{eq:trick-for-computing-vtot}
	v(K)\define\sum_{F\in \Facets{K}}v(F).
\end{equation}
With this in mind, we keep going in our computation by rewriting $v(F_3\cap (F_1\cup F_2))$ and adding $v(F_4)$ and so on. Thus, we define $\vtot^{simpl}$ as the following real number:
%Then it is natural to define $v_{tot}$ as follow:
\begin{equation}\label{eq:total-number-payoff}
	\vtot^{simpl}\define v(F_1)+v(F_2)-v(F_1\cap F_2)+v(F_3)-v(F_3\cap (F_1\cup F_2))+\dots
\end{equation}

Let us now prove that such number is well define, by showing that is is independent by the ordering of the facets.

\begin{remark}
The arithmetic trick we are using is not so far from the idea proposed in \cite{Shapley-matroids-static}. Indeed their efficiency axiom can be seen as the probabilistic (weighted) version of  $\sum_{i=1}^n \phi_i(v)=v(\Delta) \overset{\operatorname{\scriptscriptstyle \mbox{ by } \eqref{eq:trick-for-computing-vtot}}}= \sum_{F\in \FacetsD} v(F)$. %adjusted by the arithmetic trick \eqref{eq:trick-for-computing-vtot}.
\end{remark}

We recall that if $A$ is a finite set, ${A \choose l}$ is the set of subsets of $A$ of cardinality $l$.

\begin{theorem}\label{thm:vtot-grothendieck}
	Let $k$ be the number of facets of a simplicial complex $\Delta$, $k\define \#\FacetsD$.
	The total number of payoff $\vtot$ for the cooperative game $(\Delta, v)$ defined in equation \eqref{eq:total-number-payoff} is precisely
	\begin{equation}\label{eq:vtot-grothendieck}
		\vtot=\sum_{l=1}^{k}\sum_{\{F_{i_j}\}\in{\FacetsD \choose l}} (-1)^{l+1}v(F_{i_1}\cap \dots \cap F_{i_l})
	\end{equation}
%	\[
%		\vtot=\sum_{k=1}^n\sum_{F_{i_1}, \dots, F_{i_k}\in{\FacetsD \choose k}} (-1)^{k+1}v(F_{i_1}\cap \dots, \cap F_{i_k})
%	\]
	where $F_{i_j}$ is the $i_j$-th elements of the set $\{F_{i_1}, \dots, F_{i_l}\}$.

	Moreover, if $\Delta$ is the full simplex $2^{[n]}$, then $\vtot=v([n])$.
\end{theorem}
\begin{proof}
	%We want to show the statement by induction 
	%By the previous lemma we 
	%\eqref{eq:trick-for-computing-vtot}
	Let us consider the sum of the characteristic function evaluated over all facets $\sum_{F\in \FacetsD} v(F)$. Of course, we are double counting (at least!) everything that is in the intersection of two facets, and to correct our computation we take away this quantity, leading to:
	\[
		\sum_{F\in \FacetsD} v(F) - \sum_{\{F_{i_1}, F_{i_2}\} \subseteq \FacetsD} v(F_{i_1}\cap F_{i_2}).
	\]
	If a certain set $T$ appears in more than in one intersection $F_{i_1}\cap F_{i_2}$, we are subtracting $v(T)$ too many times. 
	So we should again correct our partial step by adding triple intersections:
	\[
		\sum_{F\in \FacetsD} v(F) - \sum_{\{F_{i_1}, F_{i_2}\} \subseteq \FacetsD} v(F_{i_1}\cap F_{i_2}) + \sum_{\{F_{i_1}, F_{i_2}, F_{i_3} \} \subseteq \FacetsD} v(F_{i_1}\cap F_{i_2}\cap F_{i_3}). 
	\]
	Because we only consider finite simplicial complex (see Section \ref{sec:preliminaries}), by iteration we get the formula in the statement.
	
	Finally, in the case $\Delta$ is the full simplex, there is only one facet, $[n]$, and so there are no double nor triple intersections. Naturally, the entire argument reduces to $\vtot=\sum_{F\in \{[n]\}} v(F)=v([n])$.	
\end{proof}

Another point of view for the previous proof is provide by considering the following sub-complexes of $\Delta$:
\begin{eqnarray*}
	\Delta^{(0)}&=&\Delta\\
	\Delta^{(1)}& \overset{\operatorname{\scriptscriptstyle not}}=& \Delta' \define \bigcup_{\{F_{i_1}, F_{i_2}\} \subseteq \FacetsD} F_{i_1}\cap F_{i_2}	
\end{eqnarray*}
and generically
\begin{equation}\label{eq:def-delta-j}
	\Delta^{(j)}= \bigcup_{\{F_{i_1}, F_{i_2},\dots, F_{i_j}\} \subseteq \FacetsD} F_{i_1}\cap F_{i_2}\cap \dots \cap F_{i_j}
\end{equation}
With this notation and using the arithmetic trick in equation \eqref{eq:trick-for-computing-vtot}, $\vtot$ is the sum with signs of the worth each of these $\Delta^{(j)}$:
\begin{equation}\label{eq:total-number-payoff-as-sum}
	\vtot=v(\Delta^{(0)})-v(\Delta^{(1)})+v(\Delta^{(2)})+\dots +(-1)^j v(\Delta^{(j)}) +\dots v(\Delta^{(\# \FacetsD)})
\end{equation}
where some of the $\Delta^{(j)}$ can be empty.

\begin{remark}\label{rmk:connection-2-matroids-probabilistic}
	If we want to take in consideration a probabilistic approach, we could provide a probability distribution for the facets of $\Delta^{(0)}$ as done in Section 4 of \cite{Shapley-matroids-static} and generalized in \cite{Martino-cooperative}. Then one could do the same for $\Delta^{(1)}$ and generically $\Delta^{(j)}$ and obtain a probabilistic version of \eqref{eq:total-number-payoff-as-sum} and so a probabilistic version of the simplicial efficiency axiom. 
	A coherent choice for such probabilities should be requested.
\end{remark}

\begin{remark}\label{rmk:connection-2-matroids-approximation}
	Another way of looking the probabilistic efficiency proposed in \cite{Shapley-matroids-static} in view of our results is the following.
%
%\noindent
%Taking aside the probabilistic approach that we have taken in consideration in the previous remark, 
Equation \eqref{eq:total-number-payoff-as-sum} shows how $v(\Delta)$ (that essentially is the condition requested in \cite{Shapley-matroids-static} seems a first approximation of the \emph{Combinatorics} of the problem.

\noindent
Nevertheless, in the matroidal case, such approximation is not so far from the Simplicial efficiency axiom. We are going to treat this in the Section \ref{subsection:matroid-refinement}.
\end{remark}

The following functions encode the coefficients of $v(T)$ in formula \eqref{eq:vtot-grothendieck}.
For every $T\in \Delta$, we denote by $d_T$ the following number:
\[
	d_T \define \left(\sum_{k, T=F_1\cap \dots \cap F_k} (-1)^{k+1}\right)
\]
where the sum runs over all positive number $k$ such that $T$ can be written as a $k$-intersection of facets of $\Delta$. It is useful to consider $d_T=0$, if $T$ cannot be written as intersection of facets.

%\begin{example}
%	Write this.
%\end{example}

\begin{theorem}\label{thm:s.e.a.}
Let $\Delta$ be a simplicial complex and let $\mathfrak{I}$ be a cone of cooperative games defined on $\Delta$ containing the carrier games $\mathcal{C}$ and $\hat{\mathcal{C}}$.

%\noindent
Let $\phi$ be a group value on $\mathcal{I}$ such that for each $i\in [n]$ and each $v\in\mathcal{I}$, we can write:
\[
		\phi_i(v)=\sum_{T\in \Link{i}{\Delta}}p_T^i \left(v(T\cup i) - v(T)\right).
\]
%(\textbf{move above})
%For every $T\in \Delta$, we denote by $d_T$ the following number:
%\[
%	d_T \define \left(\sum_{k, T=F_1\cap \dots \cap F_k} (-1)^{k+1}\right)
%\]
%where the sum runs over all positive number $k$ such that $T$ can be written as a $k$-intersection of facets of $\Delta$. It is useful to consider $d_T=0$, if $T$ cannot be written as intersection of facets.
%(\textbf{move above})

\noindent
The group value $\phi$ satisfies the simplicial efficiency axiom if and only if for all non-facet $T$ in $\Delta$
\begin{equation}\label{eq:first-condition-efficiency-simplicial}
	\sum_{i\in T}p^i_{T\setminus i}-\sum_{j, T\in \Link{j}{\Delta}}p_T^j=d_T,
\end{equation}
and for all facet $F$ of $\Delta$
\begin{equation}\label{eq:second-condition-efficiency-simplicial}
	\sum_{i\in F}p^i_{F\setminus i}=1.
\end{equation}
\end{theorem}
\begin{proof}
	Apply Theorem \ref{thm:g.e.a.} with $a_T=d_T$
%	Let us show the if-part.
%	Let $\sum\phi$ be the sum of the individual value $\phi_i(v)$ for all players. Using our assumption, this is:
%	%We want to compute $\sum\phi(v)$ that is $\sum_i \phi_i(v)$ and by our assumptions this is
%	\[
%		\sum\phi=\sum_{i\in [n]} \sum_{T\in \Link{i}{\Delta}}p_T^i \left(v(T\cup i) - v(T)\right).
%	\]
%	We reorder the terms in the sum as
%	\[
%		\sum\phi=\sum_{T\in \Delta}v(T)\left(\sum_{i\in T}p^i_{T\setminus i}-\sum_{j, T\in \Link{j}{\Delta}}p_T^j\right).
%	\]
%	Now we also reorder the summands in \eqref{eq:vtot-grothendieck} as
%	\[
%		\vtot=\sum_{T\in \Delta}v(T) \left(\sum_{k, T=F_1\cap \dots \cap F_k} (-1)^{k+1}\right),
%	\]
%	and this is precisely
%	\[
%		\vtot=\sum_{T\in \Delta}v(T) d_T.
%	\]
%
%
%	If $T$ is a facet, then $d_T=1$ and therefore we get \eqref{eq:second-condition-efficiency-simplicial-corrected}.
%	If $T$ is not a facet, then it is clear that if $\phi$ satisfies \eqref{eq:first-condition-efficiency-simplicial-corrected}.
%	
%%	Staring at the same equations, but on the opposite direction shows also the only if part.
%	
%	To prove the opposite direction one needs to note that if $T$ is not a facet, then
%	\[
%		\sum\phi(v_T)-\sum\phi(\hat{v}_T)=	\sum_{i\in T}p^i_{T\setminus i}-\sum_{j, T\in \Link{j}{\Delta}}p_T^j.
%	\]
%	By hypothesis the latter equals $d_T$.
%	
%	Similarly if $F$ is a facet then
%	\[
%		\sum\phi(v_F) = \sum_{i\in T}p^i_{T\setminus i}.
%	\]
%	and by hypothesis the latter is equal to zero.
\end{proof}

\subsection{The matroid case}\label{subsection:matroid-refinement}
In the case the simplicial complex is a matroid $\Delta=M$ as treated in \cite{Shapley-matroids-dynamic, Shapley-matroids-static, MR3886659,MR2847360, MR2825616}, then all the facets (bases) $\{B_{i_1},\dots, B_{i_k}\}$ have the same cardinality, say $r$. 

Therefore by using \eqref{eq:trick-for-computing-vtot}, we prove that the total number of payoff is completely determined as in equation \eqref{eq:total-number-payoff-as-sum} by the $\Delta^{(0)}$ and $\Delta^{(1)}$, that is, by the payoff of the facets and of the intersections of cardinality $r-1$. (We have simplified out notation by denoting $\Delta^{(1)}$ as $\Delta'$.) % obtained by considering only elements of dimension $r-1$ ans $r-2$ 

\vspace{0.5cm}

First let us rewrite \eqref{eq:total-number-payoff} and set up some notations.
Consider a random order of the bases $B_1, B_2, ..., B_k$.
The notation $B_0\define\emptyset$ is useful.
Let us denote
\[
	\tilde{B_j}\define B_0\cup B_1\cup\dots \cup B_j.
\]
This is the sequential partial union of the facets under the given order.
For instance, $\tilde{B_0}=\emptyset$, $\tilde{B_1}=B_1$, $\tilde{B_2}=B_1\cup B_2$ and $\tilde{B_k}=\Delta$.

Matroids are simplicial complex with a special property: indeed they are shellable \cite{MR570784, MR1333388, MR1401765, MR3558045, MR3638330, MR3649231} and, therefore, there exists a shelling order of the facets (bases) $B_1, B_2, ..., B_k$ such that $\tilde{B}_{j-1}\cap B_j$ has codimension $1$, that is $\tilde{B}_{j-1}\cap B_j$ has dimension $r-2$ (the intersection is made by cardinality $r-1$ faces). 

\begin{remark}
The dimension of a simplicial complex differ by one with respect its rank. For instance, every non-empty graphs have rank 2, because every edge is identified by two verticies, and graphs are one dimensional.
\end{remark}

%\begin{lemma}\label{lem:total-number-payoff-partial-union}
%Consider a random order of the facets $F_1, F_2, ..., F_k$ of a simplicial complex and denote by $\tilde{F}_j$ the sequential partial union of these.
%
%The total number of payoff for a cooperative game on a simplicial complex $(\Delta, v)$ is the following number:
%\[
%	\vtot(\Delta) = \sum_{j=1}^q \left(v(F_j)-v(\tilde{F}_{j-1}\cap F_j)\right).
%\]
%The total number of payoff $\vtot$ is independent of the order of the facets.
%\end{lemma}
%\begin{proof}
%	The proof is straightforward.
%\end{proof}

\begin{theorem}\label{thm-reduction-to-matroids}
	Let $(M, v)$ be a cooperative game on a matroid $M$ of rank $r$. Then, there exists an ordering of the facets (the shelling order of the bases) such that the total number of payoff is provided as
	\[
		\vtot=v(\Delta)-v(\Delta')
	\]	
	where 
	$$\Delta'= \bigcup_{\{F_{i_1}, F_{i_2}\} \subseteq \FacetsD} F_{i_1}\cap F_{i_2}.$$

	In other words,
	\[
		\vtotof{M} = \sum_{B\in \FacetsD}^q v(B) -\sum_{L} v(L),
	\]
	where the second sum runs over the subcomplex $L$ of $\Delta$ of dimension $r-2$ that are written as the intersection $\tilde{F}_{j-1}\cap F_j$ for $j=1,\dots, k$.
\end{theorem}
\begin{proof}
	This is proved using shellability together with  Theorem \ref{thm:vtot-grothendieck} and equation \eqref{eq:total-number-payoff-as-sum}.
\end{proof}

 	\vspace{0.5cm}

\bibliographystyle{amsalpha}%ma potrebbe essere migliore.
\bibliography{unica}

\def\cprime{$'$}
\providecommand{\bysame}{\leavevmode\hbox to3em{\hrulefill}\thinspace}
\providecommand{\MR}{\relax\ifhmode\unskip\space\fi MR }
\providecommand{\MRhref}[2]{%
  \href{http://www.ams.org/mathscinet-getitem?mr=#1}{#2}
}
\providecommand{\href}[2]{#2}
\begin{thebibliography}{BDJLL02}

\bibitem[AB17]{MR3638330}
Karim~A. Adiprasito and Bruno Benedetti, \emph{Subdivisions, shellability, and
  collapsibility of products}, Combinatorica \textbf{37} (2017), no.~1, 1--30.
  \MR{3638330}

\bibitem[ACS16]{MR3558045}
Federico Ardila, Federico Castillo, and Jos\'{e}~Alejandro Samper, \emph{The
  topology of the external activity complex of a matroid}, Electron. J. Combin.
  \textbf{23} (2016), no.~3, Paper 3.8, 20. \MR{3558045}

\bibitem[BDJLL01]{Shapley-matroids-static}
J.~M. Bilbao, T.~S.~H. Driessen, A.~Jim\'{e}nez~Losada, and E.~Lebr\'{o}n,
  \emph{The {S}hapley value for games on matroids: the static model}, Math.
  Methods Oper. Res. \textbf{53} (2001), no.~2, 333--348. \MR{1842713}

\bibitem[BDJLL02]{Shapley-matroids-dynamic}
J.~M. Bilbao, T.~S.~H. Driessen, A.~Jim\'{e}nez-Losada, and E.~Lebr\'{o}n,
  \emph{The {S}hapley value for games on matroids: the dynamic model}, Math.
  Methods Oper. Res. \textbf{56} (2002), no.~2, 287--301. \MR{1938216}

\bibitem[Bj{\"o}80]{MR570784}
Anders Bj{\"o}rner, \emph{Shellable and {C}ohen-{M}acaulay partially ordered
  sets}, Trans. Amer. Math. Soc. \textbf{260} (1980), no.~1, 159--183.
  \MR{570784}

\bibitem[BM19]{Borzi-Martino-D-matroids}
Alessio Borz{\`i} and Ivan Martino, \emph{Set of independencies and tutte
  polynomial of matroids over a domain}, \arxiv{1909.00332}, 2019.

\bibitem[BW96]{MR1333388}
Anders Bj\"{o}rner and Michelle~L. Wachs, \emph{Shellable nonpure complexes and
  posets. {I}}, Trans. Amer. Math. Soc. \textbf{348} (1996), no.~4, 1299--1327.
  \MR{1333388}

\bibitem[BW97]{MR1401765}
Anders Bj{\"o}rner and Michelle~L. Wachs, \emph{Shellable nonpure complexes and
  posets. {II}}, Trans. Amer. Math. Soc. \textbf{349} (1997), no.~10,
  3945--3975. \MR{1401765}

\bibitem[FV11]{MR2825616}
Ulrich Faigle and Jan Voss, \emph{A system-theoretic model for cooperation,
  interaction and allocation}, Discrete Appl. Math. \textbf{159} (2011),
  no.~16, 1736--1750. \MR{2825616}

\bibitem[GM16]{MR3503390}
Ornella Greco and Ivan Martino, \emph{Syzygies of the {V}eronese modules},
  Comm. Algebra \textbf{44} (2016), no.~9, 3890--3906. \MR{3503390}

\bibitem[GM18]{Martino-Greco-pinched}
\bysame, \emph{Cohen-macaulay property and linearity of pinched veronese
  rings}, \arxiv{1709.10461}, 2018.

\bibitem[Mar15]{Martino2015e}
Ivan Martino, \emph{Vertex collapsing and cut ideals}, Serdica Math. J.
  \textbf{41} (2015), no.~2-3, 229--242. \MR{3363603}

\bibitem[Mar18]{Martino2018}
\bysame, \emph{Face module for realizable {$\Bbb Z$}-matroids}, Contrib.
  Discrete Math. \textbf{13} (2018), no.~2, 74--87. \MR{3897225}

\bibitem[Mar20a]{Martino-cooperative}
\bysame, \emph{Cooperative games on simplicial complexes}, \arxiv{2001.00366},
  2020.

\bibitem[Mar20b]{Martino-Probabilistic-value}
\bysame, \emph{Probabilistic values for simplicial complexes},
  \arxiv{2001.05820}, 2020.

\bibitem[Moc12]{Moci2012}
Luca Moci, \emph{Wonderful models for toric arrangements}, Int. Math. Res. Not.
  IMRN (2012), no.~1, 213--238. \MR{2874932}

\bibitem[MTMZ19]{MR3886659}
Fanyong Meng, Jie Tang, Beiling Ma, and Qiang Zhang, \emph{Proportional
  coalition values for monotonic games on convex geometries with a coalition
  structure}, J. Comput. Appl. Math. \textbf{348} (2019), 34--47. \MR{3886659}

\bibitem[MZ11]{MR2847360}
Fanyong Meng and Qiang Zhang, \emph{The fuzzy core and {S}hapley function for
  dynamic fuzzy games on matroids}, Fuzzy Optim. Decis. Mak. \textbf{10}
  (2011), no.~4, 369--404. \MR{2847360}

\bibitem[NZKI97]{MR1436577}
Hiroshi Nagamochi, Dao-Zhi Zeng, Naohisa Kabutoya, and Toshihide Ibaraki,
  \emph{Complexity of the minimum base game on matroids}, Math. Oper. Res.
  \textbf{22} (1997), no.~1, 146--164. \MR{1436577}

\bibitem[Oxl11]{Oxley}
James Oxley, \emph{Matroid theory}, second ed., Oxford Graduate Texts in
  Mathematics, vol.~21, Oxford University Press, Oxford, 2011. \MR{2849819}

\bibitem[Sha53]{Shapley-a-value}
L.~S. Shapley, \emph{A value for {$n$}-person games}, Contributions to the
  theory of games, vol. 2, Annals of Mathematics Studies, no. 28, Princeton
  University Press, Princeton, N. J., 1953, pp.~307--317. \MR{0053477}

\bibitem[Sha72]{Shapley-core-convex}
Lloyd~S. Shapley, \emph{Cores of convex games}, Internat. J. Game Theory
  \textbf{1} (1971/72), 11--26; errata, ibid. 1 (1971/72), 199. \MR{311338}

\bibitem[Sta84]{MR782306}
Richard~P. Stanley, \emph{An introduction to combinatorial commutative
  algebra}, Enumeration and design ({W}aterloo, {O}nt., 1982), Academic Press,
  Toronto, ON, 1984, pp.~3--18. \MR{782306}

\bibitem[Sta91]{Stanley1991a}
\bysame, \emph{{$f$}-vectors and {$h$}-vectors of simplicial posets}, J. Pure
  Appl. Algebra \textbf{71} (1991), no.~2-3, 319--331. \MR{1117642}

\bibitem[Sta96]{Stanley1996a}
\bysame, \emph{Combinatorics and commutative algebra}, second ed., Progress in
  Mathematics, vol.~41, Birkh\"auser Boston, Inc., Boston, MA, 1996.
  \MR{1453579}

\bibitem[Sta12]{Stanley2012b}
\bysame, \emph{Enumerative combinatorics. {V}olume 1}, second ed., Cambridge
  Studies in Advanced Mathematics, vol.~49, Cambridge University Press,
  Cambridge, 2012. \MR{2868112}

\bibitem[SW17]{MR3649231}
Jay Schweig and Russ Woodroofe, \emph{A broad class of shellable lattices},
  Adv. Math. \textbf{313} (2017), 537--563. \MR{3649231}

\bibitem[Web88]{Weber-robabilistic-values-for-games}
Robert~James Weber, \emph{P}, The {S}hapley value, Cambridge Univ. Press,
  Cambridge, 1988, pp.~101--119. \MR{989825}

\bibitem[Whi35]{Whi35}
Hassler Whitney, \emph{On the {A}bstract {P}roperties of {L}inear
  {D}ependence}, Amer. J. Math. \textbf{57} (1935), no.~3, 509--533.

\bibitem[Zha99]{MR1707975}
Jingang Zhao, \emph{A necessary and sufficient condition for the convexity in
  oligopoly games}, Math. Social Sci. \textbf{37} (1999), no.~2, 189--204.
  \MR{1707975}

\end{thebibliography}

 	\vspace{0.5cm}
	
 	\noindent
 	{\scshape Ivan Martino}\\
 	{\scshape Department of Mathematics, Royal Institute of Technology.}\\ %Stockholm, Sweden}\\
 	{\itshape E-mail address}: \texttt{imartino@kth.se}
\end{document}